\documentclass{amsart}
\usepackage{amsfonts}
\usepackage{amsmath,amsthm}
\usepackage{amsfonts,amssymb}
\usepackage{multirow}
\usepackage{booktabs}
\usepackage{enumerate}

\newtheorem{thm}{Theorem}[section]
\newtheorem{cor}[thm]{Corollary}
\newtheorem{lem}[thm]{Lemma}

\newtheorem{rem}[thm]{Remark}

\numberwithin{equation}{section}

\newcommand{\al}{\alpha}

\def\vz{\varepsilon}
\def\oz{\omega}
\def\lz{\lambda}
\def\Lz{\Lambda}

\def\dz{\delta}

\def\Gz{\Gamma}

\def\({\Bigl(}
\def \){ \Bigr)}

 \def\RR{{\mathbb R}}

\def\va{\varepsilon}

\begin{document}
\def\RR{\mathbb{R}}
\def\Exp{\text{Exp}}
\def\FF{\mathcal{F}_\al}

\title[] {On the power of standard information for tractability for
$L_2$-approximation in the average case  setting}

\author{Wanting Lu} \address{ School of Mathematical Sciences, Capital Normal
University, Beijing 100048,
 China.}
 \email{luwanting1234@163.com}

\author{Heping Wang} \address{ School of Mathematical Sciences, Capital Normal
University,
Beijing 100048,
 China.}
\email{wanghp@cnu.edu.cn}

\keywords{Tractability, Standard information, general linear
information, Randomized setting} \subjclass[2010]{41A63; 65C05;
65D15;  65Y20}

\begin{abstract} We study multivariate approximation  in the average case setting  with
the error measured in the weighted $L_2$ norm. We consider
algorithms that use standard information $\Lz^{\rm std}$
consisting of function values or general linear information
$\Lz^{\rm all}$ consisting of arbitrary continuous linear
functionals.   We investigate the equivalences of various notions
of algebraic
 and exponential tractability  for
$\Lz^{\rm std}$  and $\Lz^{\rm all}$  for the absolute  error
criterion, and show that  the power of $\Lz^{\rm std}$  is the
same as that of $\Lz^{\rm all}$  for all notions of algebraic
 and exponential tractability  without any condition. Specifically, we solve
Open Problems 116-118 and almost solve Open Problem 115 as posed
by E.Novak and H.Wo\'zniakowski in the book: Tractability of
Multivariate Problems, Volume III: Standard Information for
Operators, EMS Tracts in Mathematics, Z\"urich, 2012.

\end{abstract}

\maketitle
\input amssym.def

\section{Introduction}

 In this paper, we study
multivariate approximation ${\rm APP}=\{{\rm APP}_d\}_{d\in\Bbb
N}$ in the average case setting, where
$${{\rm APP}}_d: F_d\to G_d\ \ {\rm with}\ \  {\rm APP}_d\,
f=f $$ is the continuous embedding operator,  $F_d$ is  a
separable Banach  function space on $D_d$ equipped with a
zero-mean Gaussian measure $\mu_d$, $G_d$ is a weighted $L_2$
space on $D_d$, $D_d$ is a Lebesgue measurable subset of $\Bbb
R^d$, and the dimension $d$ is large or even huge. We consider
algorithms that use finitely many information evaluations. Here
information evaluation means continuous linear functional on $F_d$
(general linear information) or function value  at some point
(standard information). We use $\Lz^{\rm all}$ and $\Lz^{\rm std}$
to denote the class of all continuous linear functionals and the
class of all function values, respectively.

 For a given error threshold
$\vz\in(0,1)$, the information complexity $n(\vz, d)$ is defined
to be the minimal number of information evaluations for which the
average case  error of some algorithm is at most $\vz$.
Tractability is aimed at studying how the information complexity
$n(\vz,d)$ depends on $\vz$ and $d$. There are two kinds of
tractability based on polynomial convergence and exponential
convergence. The algebraic tractability (ALG-tractability)
describes how the information complexity $n(\vz, d)$ behaves as a
function of $d$ and $\vz^{-1}$,  while the exponential
tractability (EXP-tractability)
 does as one of $d$ and $(1+\ln\vz^{-1})$. The existing notions of
tractability mainly include strong polynomial tractability (SPT),
polynomial tractability (PT), quasi-polynomial tractability (QPT),
weak tractability (WT), $(s, t)$-weak tractability ($(s,t)$-WT),
and uniform weak tractability (UWT).
 In
recent years   the study of algebraic  and
 exponential tractability has attracted much interest, and a great number of
interesting results
 have been obtained
  (see
\cite{NW1, NW2, NW3, W, GW, X1, S1, SiW,  DLPW,  DKPW, KW, PP, X2,
IKPW, CW, LX2, PPXD} and the references therein).

This paper is devoted to investigating  the equivalences of
various notions of algebraic and exponential tractability  for
 $\Lz^{\rm std}$ and $\Lz^{\rm all}$
 in the average case   setting (see \cite[Chapter
24]{NW3}). The class $\Lz^{\rm std}$ is much smaller and much more
practical,  and is much more difficult to analyze than the class
$\Lz^{\rm all}$. Hence, it is very important to study the power of
$\Lz^{\rm std}$ compared to $\Lz^{\rm all}$. There are many paper
devoted to this field. For example, for the randomized setting,
see \cite{NW3, WW07, KWW, K19, CDL, KWW,  CM, LW}; for the average
case setting, see \cite{NW3, HWW, LZ, X5}; for the worst case
setting, see \cite{NW3, WW01, HNV, KWW1, NW16, NW17, KU, KU2, KUV,
NSU, KS}.

In \cite{HWW, NW3, X5} the authors obtained the equivalences of
various notions of algebraic and exponential tractability
 in the average case setting for
 $\Lz^{\rm std}$ and $\Lz^{\rm all}$ for the normalized error
criterion  without any condition. Meanwhile,  for the absolute
error criterion under some conditions on the initial error, the
equivalences of ALG-SPT, ALG-PT, ALG-QPT, ALG-WT were also
obtained in \cite{NW3}.
 Xu obtained in \cite{X5} the equivalences of
ALG-PT, ALG-QPT, ALG-WT for $\Lz^{\rm all}$ and $\Lz^{\rm std}$
under much weaker conditions. This gives a partial solution to
Open problems 116-118 in \cite{NW3}.  Xu also obtained in
\cite{X5} the equivalences of ALG-$(s,t)$-WT, ALG-UWT, and various
notions of EXP-tractability under some conditions on the initial
error.

In this paper  we obtain the equivalences of various notions of
algebraic and exponential tractability   for $\Lz^{\rm all}$ and
$\Lz^{\rm std}$ in the average case   setting  for the absolute
error criterion without any condition, which means the above
conditions  are unnecessary. This completely solves Open problems
116-118 in \cite{NW3}.
 We also give an almost complete solution to  Open
Problem 115 in \cite{NW3}.

This paper is organized as follows.   In Subsections 2.1  we
introduce the approximation problem in the average  case setting.
The various notions of algebraic and exponential tractability are
given in Subsection 2.2.
   Our main results Theorems
  2.1-2.5
 are stated  in Subsection 2.3. In  Section 3, we give the proofs of Theorems 2.1 and
2.2. After that, in Section 4  we show  the equivalences of  the
notions of algebraic tractability for the absolute error criterion
without any condition. The equivalence results for the notions of
exponential tractability for  the absolute error criterion are
proved in Section 5.

\section{Preliminaries and Main Results}

\subsection{Average case setting}\

 For $d\in \Bbb N$, let $F_d$ be a separable Banach space
of $d$-variate real-valued functions on $D_d$ equipped with  a
zero-mean Gaussian measure $\mu_d$, $G_d=L_{2}(D_d,\rho_{d}(x)dx)$
be a weighted $L_2$ space, where $D_d$ is a Borel measurable
subset of $\mathbb{R}^{d}$ with positive Lebesgue measure,
$\rho_d$ is a probability density function on $D_d$.  We consider
 the multivariate approximation
problem ${\rm APP}=\{{\rm APP}_d\}_{d\in \Bbb N}$  in the
average case setting which is defined via the continuous linear operator
\begin{equation}\label{2.1} {{\rm APP}}_d: F_d\to G_d\ \ {\rm with}\ \  {\rm APP}_d\,
f=f.\end{equation} We suppose that   function value at some point
 $x\in D_d$  is well defined continuous linear functional on $F_d$.
That is, we suppose that  $\Lz^{\rm std}\subset \Lz^{\rm
all}=(F_d)^*$, where $(F_d)^*$ is the dual space of $F_d$.  It is
well known that, in the average case setting with the average
being with respect to a zero-mean Gaussian measure, adaptive
choice of the above information evaluations do not essentially
help, see \cite{TWW}. Hence, we can restrict our attention to
nonadaptive algorithms, i.e., algorithms $A_{n,d}f$  of the form
\begin{equation}\label{2.2}A_{n,d}f=\phi
_{n,d}(L_1(f),L_2(f),\dots,L_n(f)),\end{equation} where $L_i\in
\Lz, \ i=1,\dots,n$, $\Lz\in\{\Lz^{\rm all},\, \Lz^{\rm std}\}$,
 and $\phi _{n,d}:\;\Bbb R^n\to
G_d$ is an arbitrary measurable mapping from $\Bbb R^n$ to $G_d$.
The average case error for the algorithm $A_{n,d}$ of the form
\eqref{2.2} is defined as
 $$e^{\rm avg}(A_{n,d}):=\Big(\int_{F_d}\|{\rm APP}_{d} f-A_{n,d}f\|^2_{G_d}\ \mu_d(df)\Big)^{1/2}.$$
  The $n$th  minimal average case error for $\Lz\in \{\Lz^{\rm all},\, \Lz^{\rm std}\}$ is defined by
$$e^{\rm avg}(n,d;\Lz):=\inf_{A_{n,d} \ {\rm with}\ L_i\in \Lz}e^{\rm avg}(A_{n,d}),$$
where the infimum is taken over all algorithms of the form
\eqref{2.2}.

For $n=0$, we use $A_{0,d}=0$. We obtain  the so-called initial
error $e^{\rm avg}(0,d)$ defined by
$$e^{\rm avg}(0,d ):=e^{\rm avg}(0,d;\Lz^{\rm
all} )=e^{\rm avg}(0,d;\Lz^{\rm std} )=\Big(\int_{F_d}\|{\rm
APP}_{d} f\|^2_{G_d}\ \mu_d(df)\Big)^{1/2}.$$We set
$$\Gz_d:=(e^{\rm avg}(0,d;\Lz^{\rm all} ))^2=(e^{\rm avg}(0,d;\Lz^{\rm
std} ))^2.$$

It follows from \cite[Chapter 6]{TWW} and \cite{NW3} that $e^{\rm
avg}(n,d;\Lz^{\rm all})$  are described through the eigenvalues
and the eigenvectors of the covariance
operator $C_{\nu_d}: \,G_d\to G_d$ of the induced measure $\nu_d=\mu_{d}S_d^{-1}$ of $\mu_d$. Here, $\mu_d$ is a zero-mean Gaussian measure of $F_d$, so that $\nu_d$ is a zero-mean Gaussian measure on the Borel sets of $G_d$. The operator $C_{\nu_d}$ is self-adjoint, non-negative definite, and the trace of $C_{\nu_d}$ is finite.\\
Let  $\big\{(\lz_{k,d},\eta_{k,d})\big\}_{k=1}^\infty$  denote the
eigenpairs of $C_{\nu_d}$ satisfying $$\lz_{1,d}\ge\lz_{2,d}\ge
\dots \lz_{n,d}\dots \ge 0.$$That is,
$\{\eta_{k,d}\}_{k=1}^\infty$ is an orthonormal basis in $G_d$,
and
$$C_{\nu_d}\,\eta_{k,d} =\lz_{k,d}\,\eta_{k,d},\  k\in\mathbb{N}. $$  From \cite{TWW, NW3}
 we get   that the $n$th minimal average case
error is
$$e^{\rm avg}(n,d; \Lz^{\rm all})=\big(\sum_{k=n+1}^{\infty}\lz_{k,d}\big)^{1/2},$$
and  it is achieved by the optimal algorithm
$$A_{n,d}^*f=\sum_{k=1}^n \langle f, \eta_{k,d} \rangle_{G_d}\,
\eta_{k,d}.$$That is,
\begin{equation}\label{2.2-0}e^{\rm avg}(n,d; \Lz^{\rm all})=\big(\int_{F_d}\|f-A_{n,d}^*f\|^2_{G_d}\mu_d(df)\big)^{1/2}=\big(\sum_{k=n+1}^{\infty}\lz_{k,d}\big)^{1/2}.\end{equation}
The trace of $C_{\nu_d}$ is just the square of the initial error
$e^{\rm avg}(0, d)$ given by
$${\rm trace}(C_{\nu_d})= \Gz_d=(e^{\rm avg}(0,d))^2=\int_{G_d}\|g\|^2_{G_d}\nu_d(dg)=
\sum_{k=1}^{\infty}\lz_{k,d}<\infty.$$

The  information complexity  can be studied using either the
absolute error criterion (ABS) or the normalized error criterion
(NOR). In the average case setting for $\star\in\{{\rm
ABS,\,NOR}\}$ and  $\Lz\in \{\Lz^{\rm all}, \Lz^{\rm std}\}$, we
define the information complexity $n^{ \star}(\va ,d;\Lz)$  as
\begin{equation} n^{ \star}(\va
,d;\Lz):=n^{ \rm avg, \star}(\va ,d;\Lz):=\inf\{n \ \big|\  e^{\rm
avg}(n,d;\Lz)\le \vz\, {\rm CRI}_d\}, \end{equation} where
\begin{equation*}
{\rm CRI}_d:=\left\{\begin{split}
 &\ \ 1, \qquad\, \quad\quad\text{for $\star$=ABS,} \\
 &e^{\rm avg}(0,d), \quad \text{ for $\star$=NOR}
\end{split}\right. \ \ =\ \ \left\{\begin{split}
 &\ 1,  \qquad \ \ \ \text{ for $\star$=ABS,} \\
 &(\Gz_d)^{1/2},\  \ \text{ for $\star$=NOR.}
\end{split}\right.
\end{equation*}
Since $\Lz^{\rm std}\subset \Lz^{\rm all},$ we get
\begin{equation*}e^{\rm avg}(n,d;\Lz^{\rm all})\le e^{\rm avg}(n,d;\Lz^{\rm std}).\end{equation*}
It follows that for $\star\in\{{\rm ABS,\,NOR}\}$,
\begin{equation} \label{2.4-1}n^{ \star}(\va
,d;\Lz^{\rm all})\le n^{ \star}(\va ,d;\Lz^{\rm std}).
\end{equation}

\subsection{Notions of tractability}\

In this subsection we briefly recall the various tractability
notions in the average case setting.  First we introduce all
notions of algebraic tractability. Let ${\rm APP}= \{{\rm
APP}_d\}_{d\in\Bbb N}$,
 $\star\in\{{\rm ABS,\,NOR}\}$, and
$\Lz\in \{\Lz^{\rm all}, \Lz^{\rm std}\}$. In the average case
setting for the class $\Lambda$, and for error criterion $\star$,
we say that ${\rm APP}$ is

$\bullet$ Algebraic strongly polynomially tractable (ALG-SPT) if
there exist $ C>0$ and non-negative number $p$ such that
\begin{equation}\label{2.7}n^{ \star}(\va ,d;\Lz)\leq C\varepsilon^{-p},\
\text{for all}\ \varepsilon\in(0,1).\end{equation}The exponent
ALG-$p^{ \star}(\Lz)$ of ALG-SPT is defined as the infimum of $p$
for which \eqref{2.7} holds;

$\bullet$ Algebraic polynomially tractable (ALG-PT)  if there
exist $ C>0$ and non-negative numbers $p,q$ such that
$$n^{ \star}(\va
,d;\Lz)\leq Cd^{q}\varepsilon^{-p},\ \text{for all}\
d\in\mathbb{N},\ \varepsilon\in(0,1);$$

$\bullet$ Algebraic quasi-polynomially tractable (ALG-QPT) if
there exist $ C>0$ and non-negative number $t$ such that
\begin{equation}\label{2.8}n^{ \star}(\va
,d;\Lz)\leq C \exp(t(1+\ln{d})(1+\ln{\varepsilon^{-1}})),\
\text{for all}\ d\in\mathbb{N},\
\varepsilon\in(0,1).\end{equation}The exponent ALG-$t^{
\star}(\Lz)$ of ALG-QPT  is defined as the infimum of $t$ for
which \eqref{2.8} holds;

$\bullet$ Algebraic uniformly weakly tractable (ALG-UWT)  if
$$\lim_{\varepsilon^{-1}+d\rightarrow\infty}\frac{\ln n^{\star}(\va
,d;\Lz)}{\varepsilon^{-\alpha}+d^{\beta}}=0,\ \text{for all}\
\alpha, \beta>0;$$

$\bullet$ Algebraic weakly tractable (ALG-WT) if
$$\lim_{\varepsilon^{-1}+d\rightarrow\infty}\frac{\ln n^{\star}(\va
,d;\Lz)}{\varepsilon^{-1}+d}=0;$$

$\bullet$ Algebraic $(s,t)$-weakly tractable (ALG-$(s,t)$-WT) for
fixed $s, t>0$ if
 $$\lim_{\varepsilon^{-1}+d\rightarrow\infty}\frac{\ln n^{\star}(\va
,d;\Lz)}{\varepsilon^{-s}+d^{t}}=0.$$

Clearly, ALG-$(1,1)$-WT is the same as ALG-WT. If ${\rm APP}$ is
not ALG-WT, then ${\rm APP}$  is called  intractable.

If
 the $n$th  minimal error    decays faster
than any polynomial and is exponentially convergent, then we
should study  tractability with $\vz^{-1}$ being replaced by
$(1+\ln \frac 1{\vz})$, which is called exponential tractability.
Recently, there have been many papers studying exponential
tractability (see \cite{DLPW, DKPW, X3, PPXD, KW, IKPW, CW, LX2}).

In the definitions of ALG-SPT, ALG-PT, ALG-QPT, ALG-UWT, ALG-WT,
and ALG-$(s,t)$-WT, if we replace $\frac1{\vz}$ by $(1+\ln \frac
1{\vz})$, we get the definitions of \emph{exponential strong
polynomial tractability} (EXP-SPT), \emph{exponential polynomial
tractability} (EXP-PT), \emph{exponential quasi-polynomial
tractability} (EXP-QPT), \emph{exponential uniform weak
tractability}
 (EXP-UWT), \emph{exponential weak tractability}
 (EXP-WT), and \emph{exponential $(s,t)$-weak tractability}
 (EXP-$(s,t)$-WT), respectively.  We now give the above notions of exponential tractability in
 detail.

Let ${\rm APP}= \{{\rm APP}_d\}_{d\in\Bbb N}$, $\star\in\{{\rm ABS,\,NOR}\}$, and $\Lz\in \{\Lz^{\rm
all}, \Lz^{\rm std}\}$. In the average case setting for the class
$\Lambda$, and for error criterion $\star$, we say that ${\rm
APP}$ is

$\bullet$ Exponential strongly polynomially tractable (EXP-SPT) if
there exist $ C>0$ and non-negative number $p$ such that
\begin{equation}\label{2.9}n^{ \star}(\va ,d;\Lz)\leq C(\ln\varepsilon^{-1}+1)^{p},\
\text{for all}\ \varepsilon\in(0,1).\end{equation}The exponent
EXP-$p^{\star}(\Lz)$ of EXP-SPT is defined as the infimum of $p$
for which \eqref{2.9} holds;

$\bullet$ Exponential  polynomially tractable (EXP-PT)  if there
exist $C>0$ and non-negative numbers $p,q$ such that
$$n^{\star}(\va
,d;\Lz)\leq Cd^{q}(\ln\varepsilon^{-1}+1)^{p},\ \text{for all}\
d\in\mathbb{N},\ \varepsilon\in(0,1);$$

$\bullet$ Exponential  quasi-polynomially tractable (EXP-QPT) if
there exist $C>0$ and non-negative number $t$ such that
\begin{equation}\label{2.10}n^{\star}(\va
,d;\Lz)\leq C \exp(t(1+\ln{d})(1+\ln(\ln\varepsilon^{-1}+1))),\
\text{for all}\ d\in\mathbb{N},\
\varepsilon\in(0,1).\end{equation}The exponent EXP-$t^{
\star}(\Lz)$ of EXP-QPT  is defined as the infimum of $t$ for
which \eqref{2.10} holds;

$\bullet$ Exponential  uniformly weakly tractable (EXP-UWT)  if
$$\lim_{\varepsilon^{-1}+d\rightarrow\infty}\frac{\ln n^{\star}(\va
,d;\Lz)}{(1+\ln\varepsilon^{-1})^{\alpha}+d^{\beta}}=0,\ \text{for
all}\ \alpha, \beta>0;$$

$\bullet$ Exponential  weakly tractable (EXP-WT) if
$$\lim_{\varepsilon^{-1}+d\rightarrow\infty}\frac{\ln n^{\star}(\va
,d;\Lz)}{1+\ln\varepsilon^{-1}+d}=0;$$

$\bullet$ Exponential  $(s,t)$-weakly tractable (EXP-$(s,t)$-WT)
for fixed $s,t>0$ if
 $$\lim_{\varepsilon^{-1}+d\rightarrow\infty}\frac{\ln n^{\star}(\va
,d;\Lz)}{(1+\ln\varepsilon^{-1})^{s}+d^{t}}=0.$$

\subsection{Main results}\

We shall give  main results of this paper in this subsection. We
remark that  for multivariate approximation problem results and
proofs in the average case  setting are  in full analogy with ones
in the randomized setting (see \cite{LW}). For the convenience of
the reader, we provide details of  all proofs.

The authors in \cite{HWW, NW3, X5} used the mean value theorem and
iterated  Monte Carlo methods to obtain  the relation between
$e^{\rm avg}(n,d;\Lz^{\rm std})$ and $e^{\rm avg}(n,d;\Lz^{\rm
all})$. We  use the mean value theorem and the method used in
\cite{KUV, LW} to get an  inequality between $e^{\rm
avg}(n,d;\Lz^{\rm std})$ and $e^{\rm avg}(n,d;\Lz^{\rm all})$. See
the following theorem.

\begin{thm}
Let $\delta\in(0,1)$,  $m,n\in\mathbb{N}$ be such that
$$m=\left\lfloor\frac{n}{48(\sqrt{2}\ln(2n)-\ln{\delta})}\right\rfloor.$$
Then we have
\begin{equation}\label{2.12}e^{\rm avg}(n,d;\Lz^{\rm std})\le\Big(1+\frac{4m}{n}\Big)^{\frac12}\frac{1}{\sqrt{1-\delta}}\,e^{\rm avg}(m,d;\Lz^{\rm all}),
\end{equation}where $\lfloor x\rfloor$ denotes the largest
integer not exceeding $x$.
\end{thm}

Based on Theorem 2.1, we obtain two relations between the
information complexities $n^{\star}(\varepsilon,d;\Lambda^{\rm
std})$ and $n^{\star}(\varepsilon,d;\Lambda^{\rm all})$ for
$\star\in\{{\rm ABS,\,NOR}\}$.

\begin{thm} For  $\star\in\{{\rm ABS,\,NOR}\}$ and $\oz>0$, we
have
\begin{equation}\label{2.17}n^{\star}(\varepsilon,d;\Lambda^{\rm std})\le C_\oz \Big(n^{\star}(\frac\varepsilon4,d;\Lambda^{\rm
all})+1\Big)^{1+\oz},
\end{equation}where $C_\oz$ is a positive constant depending only on $\oz$.
Similarly,  for sufficiently small $\oz,\dz>0$ and $\star\in\{{\rm
ABS,\,NOR}\}$, we have
\begin{equation}\label{2.18}
n^{\star}(\varepsilon,d;\Lambda^{\rm std}) \le C_{\oz,\dz}
 \big(n^{\star}(\frac{\varepsilon}{A_\delta},d;\Lambda^{\rm
all})+1\big)^{1+\oz},
\end{equation}where
$A_{\delta}:=\Big(1+\frac{1}{12\ln{\frac{1}{\delta}}}\Big)^{\frac12}\frac{1}{\sqrt{1-\delta}}$,
$C_{\oz,\dz}$ is a positive constant depending only on $\oz$ and
$\dz$.
\end{thm}

In the average case  setting,  for the normalized error criterion,
\cite[Theorems 24.10, 24.12, and 24.6]{NW3} gives the equivalences
of  ALG-PT (ALG-SPT), ALG-QPT, ALG-WT for $\Lz^{\rm all}$ and
$\Lz^{\rm std}$, and  shows that the exponents of ALG-SPT and
ALG-QPT for $\Lz^{\rm all}$ and $\Lz^{\rm std}$  are same;
\cite[Theorems 3.4 and 3.5]{X5} gives the equivalences of
ALG-$(s,t)$-WT, ALG-UWT for  $\Lz^{\rm all}$ and $\Lz^{\rm std}$.

For the absolute  error criterion, \cite[Theorems 24.11, 24.13,
and 24.6]{NW3}  gives the equivalences of  ALG-PT (ALG-SPT),
ALG-QPT, ALG-WT for $\Lz^{\rm all}$ and $\Lz^{\rm std}$ under some
conditions on the initial error. Novak and Wo\'zniakowski posed
Open problems 116-118 in \cite{NW3} which ask whether the above
conditions are necessary. Xu obtained in \cite[Theorems
3.1-3.5]{X5} the equivalences of ALG-PT, ALG-QPT, ALG-WT,
ALG-$(s,t)$-WT, ALG-UWT for $\Lz^{\rm all}$ and $\Lz^{\rm std}$
under much weaker conditions. This gives a partial solution to
Open problems 116-118 in \cite{NW3}.

In this paper  we obtain the equivalences of ALG-SPT, ALG-PT,
ALG-QPT, ALG-WT, ALG-$(s,t)$-WT, ALG-UWT   for $\Lz^{\rm all}$ and
$\Lz^{\rm std}$ in the average case   setting  for the absolute
error criterion without any condition, which means the above
conditions  are unnecessary. This solves Open problems 116-118 in
\cite{NW3}. See the following theorem.

\begin{thm} Consider the problem ${\rm APP}=\{{\rm APP}_d\}_{d\in \Bbb N}$
in the average case   setting for the absolute error criterion.
Then\vskip 2mm

$\bullet$ $\rm ALG$-$\rm SPT$, $\rm ALG$-$\rm PT$, $\rm ALG$-$\rm
QPT$, $\rm ALG$-$\rm WT$,  $\rm ALG$-$(s,t)$-$\rm WT$,   $\rm
ALG$-$\rm UWT$ for $\Lz^{\rm all}$ is equivalent to $\rm ALG$-$\rm
SPT$, $\rm ALG$-$\rm PT$, $\rm ALG$-$\rm QPT$, $\rm ALG$-$\rm WT$,
$\rm ALG$-$(s,t)$-$\rm WT$,   $\rm ALG$-$\rm UWT$ for $\Lz^{\rm
std}$;\vskip 2mm

$\bullet$ The exponents $\rm ALG$-$p^{\rm  ABS}(\Lz)$ of $\rm
ALG$-$\rm SPT$ for $\Lz^{\rm all}$ and $\Lz^{\rm std}$ are same,
and the exponents $\rm ALG$-$t^{\rm  ABS}(\Lz)$ of $\rm ALG$-$\rm
QPT$ for $\Lz^{\rm all}$ and $\Lz^{\rm std}$ are also same.
\end{thm}

For  exponential convergence in the average case setting, we first
give an almost complete solution to Open Problem 115 in
\cite{NW3}. \vskip 2mm

In the average case  setting for the normalized  error criterion,
Xu obtained in \cite[Theorems 4.1-4.5]{X5} the equivalences of
$\rm EXP$-$\rm SPT$, $\rm EXP$-$\rm PT$, $\rm EXP$-$\rm QPT$, $\rm
EXP$-$\rm WT$, $\rm EXP$-$(s,t)$-$\rm WT$, $\rm EXP$-$\rm UWT$ for
$\Lz^{\rm all}$ and $\Lz^{\rm std}$, however, he did not show that
the exponents of EXP-SPT and EXP-QPT for $\Lz^{\rm all}$ and
$\Lz^{\rm std}$ are same.

For the absolute  error criterion, Xu also obtained in
\cite[Theorems 4.1-4.5]{X5} the equivalences of $\rm EXP$-$\rm
SPT$, $\rm EXP$-$\rm PT$, $\rm EXP$-$\rm QPT$, $\rm EXP$-$\rm WT$,
$\rm EXP$-$(s,t)$-$\rm WT$, $\rm EXP$-$\rm UWT$ for $\Lz^{\rm
all}$ and $\Lz^{\rm std}$ under  weak conditions on the initial
error.

In this paper  we obtain the equivalences of EXP-SPT, EXP-PT,
EXP-QPT, EXP-WT, EXP-$(s,t)$-WT, EXP-UWT   for $\Lz^{\rm all}$ and
$\Lz^{\rm std}$ in the average case   setting  for the absolute
error criterion without any condition, which means the above
conditions   are unnecessary.  We also show that the exponents of
EXP-SPT and EXP-QPT for $\Lz^{\rm all}$ and $\Lz^{\rm std}$ are
same  for  the normalized or absolute error criterion. See the
following theorem.

\begin{thm}
Consider the problem ${\rm APP}=\{{\rm APP}_d\}_{d\in \Bbb N}$ in
the average case  setting. Then \vskip 2mm

$\bullet$   for the absolute   error criterion, $\rm EXP$-$\rm
SPT$, $\rm EXP$-$\rm PT$, $\rm EXP$-$\rm QPT$, $\rm EXP$-$\rm WT$,
$\rm EXP$-$(s,t)$-$\rm WT$, $\rm EXP$-$\rm UWT$  for $\Lz^{\rm
all}$ is equivalent to $\rm EXP$-$\rm SPT$, $\rm EXP$-$\rm PT$,
$\rm EXP$-$\rm QPT$, $\rm EXP$-$\rm WT$, $\rm EXP$-$(s,t)$-$\rm
WT$, $\rm EXP$-$\rm UWT$
 for $\Lz^{\rm
std}$;\vskip 2mm

$\bullet$ for $\star\in\{{\rm ABS,NOR}\}$, the exponents $\rm
EXP$-$p^{ \star}(\Lz)$ of $\rm EXP$-$\rm SPT$ for $\Lz^{\rm all}$
and $\Lz^{\rm std}$ are same, and the exponents $\rm
EXP$-$t^{\star}(\Lz)$ of $\rm EXP$-$\rm QPT$ for $\Lz^{\rm all}$
and $\Lz^{\rm std}$ are also same.

\end{thm}

Combining the obtained results in \cite{NW3, X5}  with Theorems
2.3 and 2.4 we obtain the following corollary.

\begin{cor}Consider the approximation problem ${\rm APP}= \{{\rm APP}_d\}_{d\in\Bbb N}$ for the absolute
or normalized error criterion in the average  case setting. Then
\vskip 2mm

$\bullet$ $\rm ALG$-$\rm SPT$, $\rm ALG$-$\rm PT$, $\rm ALG$-$\rm
QPT$,   $\rm ALG$-$\rm WT$,  $\rm ALG$-$(s,t)$-$\rm WT$, $\rm
ALG$-$\rm UWT$  for $\Lz^{\rm all}$ is equivalent to $\rm
ALG$-$\rm SPT$, $\rm ALG$-$\rm PT$, $\rm ALG$-$\rm QPT$, $\rm
ALG$-$\rm WT$, $\rm ALG$-$(s,t)$-$\rm WT$, $\rm ALG$-$\rm UWT$ for
$\Lz^{\rm std}$; \vskip 2mm

$\bullet$ $\rm EXP$-$\rm SPT$, $\rm EXP$-$\rm PT$, $\rm EXP$-$\rm
QPT$, $\rm EXP$-$\rm WT$, $\rm EXP$-$(s,t)$-$\rm WT$, $\rm
EXP$-$\rm UWT$  for $\Lz^{\rm all}$ is equivalent to $\rm
EXP$-$\rm SPT$, $\rm EXP$-$\rm PT$, $\rm EXP$-$\rm QPT$,  $\rm
EXP$-$\rm WT$, $\rm EXP$-$(s,t)$-$\rm WT$, $\rm EXP$-$\rm UWT$ for
$\Lz^{\rm std}$; \vskip 2mm

$\bullet$  the exponents of ${\rm SPT}$ and ${\rm QPT}$ are the
same  for  $\Lz^{\rm all}$  and  $\Lz^{\rm std}$, i.e., for $\star
\in \{{\rm ABS, NOR}\}$,
\begin{align*}{\rm ALG}\!-\!p^{\star}(\Lz^{\rm all}) &= {\rm ALG}\!-\!p^{\star}(\Lz^{\rm
std}), \qquad  {\rm ALG}\!-\!t^{\star}(\Lz^{\rm all}) = {\rm
ALG}\!-\!t^{\star}(\Lz^{\rm std}), \\
{\rm EXP}\!-\!p^{\star}(\Lz^{\rm all}) &= {\rm
EXP}\!-\!p^{\star}(\Lz^{\rm std}),\qquad  {\rm
EXP}\!-\!t^{\star}(\Lz^{\rm all}) = {\rm
EXP}\!-\!t^{\star}(\Lz^{\rm std}).
\end{align*}

\end{cor}

\section{Proofs of Theorems 2.1 and 2.2}

Let us keep the notation of  Subsection 2.1. For any
$m\in\mathbb{N}$, we define the  functions $ h_{m,d}(x)$ and
$\oz_{m,d}$ on $D_d$ by
$$h_{m,d}(x):=\frac{1}{m}\sum_{j=1}^{m}|\eta_{j,d}(x)|^{2},\ \ \
\oz_{m,d}(x):=h_{m,d}(x)\,\rho_d(x),$$where
$\{\eta_{j,d}\}_{j=1}^\infty$ is an orthonormal basis in
$G_d=L_2(D_d,\rho_d(x)dx)$.  Then $\oz_{m,d}$ is a probability
density function on $D_d$, i.e., $\int_{D_d}\oz_{m,d}(x)\,dx=1$.
We define the corresponding probability measure $\mu_{m,d}$ by
$$\mu_{m,d}(A)=\int_{A}\oz_{m,d}(x)\,dx,$$where $A$ is a Borel subset of
$D_d$. We use the convention that $\frac00:=0$. Then $\{\tilde
\eta_{j,d}\}_{j=1}^\infty$ is an orthonormal system in
$L_2(D_d,\mu_{m,d})$, where $$ \tilde \eta_{j,d}:=\frac
{\eta_{j,d}} {\sqrt{h_{m,d}}} .$$

For ${\rm X}=(x^1, \dots, x^n)\in D_d^n$, we use the following
matrices
\begin{equation}
\widetilde{L}_m=\widetilde{L}_m({\rm X})=\left(
\begin{array}{cccc}
\widetilde{\eta}_{1,d}(x^{1})&\widetilde{\eta}_{2,d}(x^{1})&\cdots&\widetilde{\eta}_{m,d}(x^{1})\\
\widetilde{\eta}_{1,d}(x^{2})&\widetilde{\eta}_{2,d}(x^{2})&\cdots&\widetilde{\eta}_{m,d}(x^{2})\\
\vdots&\vdots&\ &\vdots\\
\widetilde{\eta}_{1,d}(x^{n})&\widetilde{\eta}_{2,d}(x^{n})&\cdots&\widetilde{\eta}_{m,d}(x^{n})\\
\end{array}
\right)\ \ \ \ \ {\rm and} \ \ \ \ \ \widetilde H_m=\frac
1n\widetilde{L}_m^*\widetilde{L}_m,\label{3.1}
\end{equation}where $A^*$ is the conjugate transpose of a matrix $A$.
Note that $$\widetilde{N}(m):=\sup\limits_{x\in
D_d}\sum\limits_{k=1}^{m}|\widetilde{\eta}_{k,d}(x)|^{2}=m.$$

According to  \cite[Propositions 5.1 and 3.1]{KUV} we have the
following results.

\begin{lem}
Let $n,m\in\mathbb{N}$. Let $x^{1},\ldots,x^{n}\in D_d$ be drawn
independently and identically distributed at random with respect
to the probability measure $\mu_{m,d}$. Then it holds that
$$\mathbb{P}(\|\widetilde{H}_{m}-I_{m}\|>\frac12)\leq(2n)^{\sqrt{2}}\exp\left(-\frac{n}{48m}\right),
$$where $\widetilde{L}_m, \ \widetilde H_m$ are given by \eqref{3.1},  $I_m$ is the identity matrix of order
$m$, and $\|L\|$ denotes the spectral norm (i.e. the largest
singular value) of a matrix $L$. Furthermore, if
$\|\widetilde{H}_{m}-I_{m}\|\le 1/2,$ then
\begin{equation}\|(\widetilde{L}^{*}_m\widetilde{L}_m)^{-1}\|\le \frac2n.\label{3.3}\end{equation}
\end{lem}

\begin{rem}From Lemma 3.1 we immediately obtain  \begin{equation}\label{3.4}\mathbb{P}\big(\|\widetilde{H}_{m}-I_{m}\|\le 1/2\big)\ge 1-\delta\end{equation}
 if  \begin{equation} m=
\Big\lfloor\frac{n}{48(\sqrt{2}\ln(2n)-\ln{\delta})}\Big\rfloor\ge
1,\label{3.2}
\end{equation} holds, where $\lfloor x\rfloor$ denotes the largest
integer not exceeding $x$.
\end{rem}

Now let $m, n\in\Bbb N$ satisfy \eqref{3.2}, $x^1,\dots,x^n$ be
independent and identically distributed sample points  from $D_d$
that are distributed according to the probability measure
$\mu_{m,d}$, and  $\widetilde L_m, \widetilde H_m$ be given by
\eqref{3.1}. We consider the conditional distribution given the
event $\|\widetilde{H}_{m}-I_{m}\|\le 1/2$ and the conditional
expectation
$$\Bbb E(X\,\big|\ \|\widetilde{H}_{m}-I_{m}\|\le 1/2)=\frac{\int_{\|\widetilde{H}_{m}-I_{m}\|\le
1/2} X(x^1,\dots,x^n)\,d\mu_{m,d}(x^1)\dots d\mu_{m,d}(x^n)}
{\mathbb{P}\big(\|\widetilde{H}_{m}-I_{m}\|\le 1/2\big)}$$ of a
random variable $X$.

If $\|\widetilde{H}_{m}-I_{m}\|\le 1/2$ for some ${\rm
X}=(x^1,\dots,x^n)\in D_d^n$, then
$\widetilde{L}_{m}=\widetilde{L}_{m}({\rm X}) $ has the full rank.
The algorithm is a weighted least squares estimator
\begin{equation}{S}^{m}_{\rm X}\,f=\mathop{\arg\min}_{g\in
V_m}
\frac{|f(x^i)-g(x^i)|^{2}}{h_{m,d}(x^i)},\label{3.5}\end{equation}
which has a unique solution, where $V_m:={\rm
span}\{\eta_{1,d},\dots,\eta_{m,d}\}. $ It follows that
${S}^{m}_{\rm X}\,f=f$ whenever $f\in V_m$.
\\

\begin{table}[!ht]
\begin{tabular}{clllrrrr}
\toprule
\multicolumn{6}{l}{\textbf{Algorithm} \quad Weighted least squares regression.}
\\
\midrule
Input:&${\rm X}=(x^1,\dots,x^n)\in D_d^n$  &    &set of distinct sampling nodes,\\
&$\tilde {\rm
f}=\Big(\frac{f(x^1)}{\sqrt{h_{m,d}(x^1)}},\dots,\frac{f(x^n)}{\sqrt{h_{m,d}(x^n)}}\Big)^T$
&
&weighted samples of $f$ evaluted\\
& & &at the nodes from ${\rm X}$,\\
&$m\in \Bbb N$  &    &$m<n$ such that the matrix \\
& & & $\widetilde{L}_{m}:=\widetilde{L}_{m}({\rm X})$ from \eqref{3.1} has\\
& & &full (column) rank.\\
\multicolumn{6}{l}{Solve the over-determined linear system}\\

\multicolumn{6}{c}{$\widetilde{L}_{m}(\widetilde
c_{1},\cdots,\widetilde c_{m})^{T}=\tilde {\rm
f}$}\\

\multicolumn{6}{l}{via least square, i.e., compute}\\

\multicolumn{6}{c}{$(\widetilde c_{1},\cdots,\widetilde
c_{m})^{T}=(\widetilde{L}^{*}_{m}\widetilde{L}_{m})^{-1}\widetilde{L}_{m}^{*}\
\tilde {\rm
f}$.}\\

\multicolumn{6}{l}{Output: $\widetilde c=(\widetilde c_{1},\cdots,\widetilde c_{m})^{T}\in\Bbb C^m$ coefficients of
 the approximant $ S_{\rm X}^{m}(f):=\sum\limits_{k=1}^{m}\widetilde c_{k}\eta_{k,d}$ }\\

 \multicolumn{6}{r} {which is the unique solution of \eqref{3.5}.} \\

\bottomrule
\end{tabular}
\end{table}

\noindent{\it Proof of Theorem 2.1. }

We use the above notation. Let  $m, n\in\Bbb N$ satisfy
\eqref{3.2}, $x^1,\dots,x^n$ be independent and identically
distributed sample points  from $D_d$ that are distributed
according to the probability measure $\mu_{m,d}$,
$\|\widetilde{H}_{m}-I_{m}\|\le 1/2$, and $S_{\rm X}^{m}(f)$ be
defined as above.
 We estimate $\|f-
S_{\rm X}^{m}(f)\|_{G_d}^2$ for $f\in F_d$.
 We set $$H_d=L_2(D_d, \mu_{m,d}).$$
We recall that $\{\eta_{j,d}\}_{j=1}^\infty$ is an orthonormal
basis in $G_d=L_2(D_d,\rho_d(x)dx)$, and hence $\{\tilde
\eta_{j,d}\}_{j=1}^\infty$ is an orthonormal system in
$H_d=L_2(D_d,\mu_{m,d})$, where $$ \tilde \eta_{j,d}:=\frac
{\eta_{j,d}} {\sqrt{h_{m,d}}},\ \ \langle  \tilde \eta_{j,d},
 \tilde \eta_{k,d}\rangle_{H_d}=\langle \eta_{j,d},\eta_{k,d}, \rangle_{G_d}=\dz_{i,j}.$$

For $f\in F_d\subset G_d$, we have
$$f=\sum_{k=1}^\infty \langle f,
\eta_{k,d} \rangle_{G_d}\, \eta_{k,d}.$$ We note that
$f-A_{m,d}^*(f)$ is orthogonal to the space $V_m$ with respect to
the inner product $\langle\cdot,\cdot\rangle_{G_d}$, and
$$A_{m,d}^*(f)-S^{m}_{\rm X}(f)=S^{m}_{\rm X}(f-A_{m,d}^*(f))\in V_m:={\rm span}\{\eta_{1,d},\dots, \eta_{m,d}\}, $$
where $$A_{m,d}^*(f)=\sum\limits_{k=1}^{m}\langle
f,\eta_{k,d}\rangle_{G_d}\eta_{k,d}.$$ It follows that
\begin{align*}\|f-S_{\rm
X}^{m}(f)\|_{G_d}^2&=\|f-A_{m,d}^*(f)\|_{G_d}^2+\|S_{\rm
X}^{m}(f-A_{m,d}^*(f))\|_{G_d}^2\\ &=\|g\|_{G_d}^2+\|S_{\rm
X}^{m}(g)\|_{G_d}^2,\end{align*}where $g:=f-A_{m,d}^*(f)$.

We recall that
$$S_{\rm X}^m(g)=\sum_{k=1}^m\widetilde{c}_k\eta_{k,d},\,\widetilde{c}=(\widetilde{c}_1,\dots,\widetilde{c}_m)^T
=(\widetilde{L}_{m}^{*}\widetilde{L}_{m})^{-1}(\widetilde{L}_{m})^{*}\widetilde{\rm
g},$$where
$$\widetilde{\rm g}:=(\widetilde{g}(x^1),\cdots,\widetilde{g}(x^n))^T,\ \ \ \widetilde{g}:=\frac{g}{\sqrt{h_{m,d}}}.$$
 Since $\{\eta_{k,d}\}_{k=1}^{\infty}$ is an orthonormal
system in $G_d$, we get
\begin{align*}
\|S_{\rm X}^{m}(g)\|_{G_d}^2=\|\widetilde{c}\|_{2}^2&=
\|((\widetilde{L}_{m})^{*}\widetilde{L}_{m})^{-1}(\widetilde{L}_{m})^{*}\widetilde{\rm g}\|_{2}^2\\
&\leq\|((\widetilde{L}_{m})^{*}\widetilde{L}_{m})^{-1}\|\cdot\|(\widetilde{L}_{m})^{*}\widetilde{\rm g}\|_{2}^{2}\\
&\leq\frac{4}{n^{2}}\|(\widetilde{L}_{m})^{*}\widetilde{\rm g}\|_{2}^{2},
\end{align*}
where $\|\cdot\|_2$ is the Euclidean norm of a vector.  We have
\begin{align*}
\|(\widetilde{L}_{m})^{*}\widetilde{\rm
g}\|_{2}^{2}&=\sum_{k=1}^{m}\Big|\sum_{j=1}^{n}\overline{\widetilde{\eta}_{k,d}(x^{j})}\cdot
\widetilde{g}(x^{j})\Big|^{2}\\
&=\sum_{k=1}^{m}\sum_{j=1}^{n}\sum_{i=1}^{n}\overline{\widetilde{\eta}_{k,d}(x^{j})}
\widetilde{g}(x^{j})\widetilde{\eta}_{k,d}(x^{i})\overline{\widetilde{g}(x^{i})}.
\end{align*}
It follows that
\begin{align*}
J&=\int_{\|\widetilde{H}_{m}-I_{m}\|\leq\frac{1}{2}}\|(\widetilde{L}_{m})^{*}\widetilde{\rm g}\|_{2}^{2}\,d\mu_{m,d}(x^{1})\ldots d\mu_{m,d}(x^{n})\\
&\le\int_{D_d^n}\|(\widetilde{L}_{m})^{*}\widetilde{\rm g}\|_{2}^{2}\,d\mu_{m,d}(x^{1})\ldots d\mu_{m,d}(x^{n})\\
&\le\sum_{k=1}^{m}\sum_{i,j=1}^{n}\int_{D_d^n}\overline{\widetilde{\eta}_{k,d}(x^{j})}
\widetilde{g}(x^{j})\widetilde{\eta}_{k,d}(x^{i})\overline{\widetilde{g}(x^{i})}
\,d\mu_{m,d}(x^{1})\ldots d\mu_{m,d}(x^{n}),
\end{align*}
Noting that for  $i\neq j$ and $1\le k\le m$,
\begin{align*}\int_{D_d^n}&\overline{\widetilde{\eta}_{k,d}(x^{j})}
\widetilde{g}(x^{j})\widetilde{\eta}_{k,d}(x^{i})\overline{\widetilde{g}(x^{i})}
\,d\mu_{m,d}(x^{1})\ldots d\mu_{m,d}(x^{n})\\
&=|\langle\widetilde{g},\widetilde{\eta}_{k,d}\rangle_{H_d}|^2
=|\langle g,\eta_{k,d}\rangle_{G_d}|^2=0,\end{align*} and
$h_{m,d}(x)=\frac{1}{m}\sum\limits_{k=1}^{m}|\eta_{k,d}(x)|^{2}$,
we continue to get \begin{align*}
J&\le n\sum_{k=1}^m\|\widetilde{\eta}_{k,d}\cdot\widetilde{g}\|_{H_d}^2\\
&=n\sum_{k=1}^m\int_{D_d^n}|\widetilde{g}(x)\widetilde{\eta}_{k,d}(x)|^2
\rho_d(x)h_{m,d}(x)\,dx\\
&=n\sum_{k=1}^m\int_{D_d^n}\frac{|g(x)\eta_{k,d}(x)|^2}{h_{m,d}(x)}
\rho_d(x)\,dx\\
&=n\int_{D_d^n}m|g(x)|^2\rho_d(x)\,dx\\
&=nm\cdot\|g\|_{G_d}^2.
\end{align*}
Hence,  we have
\begin{align} &\quad\ \int_{\|\widetilde{H}_{m}-I_{m}\|\leq\frac{1}{2}}\|f-S_X^m(f)\|_{G_d}^{2}\,d\mu_{m,d}(x^{1})\ldots
d\mu_{m,d}(x^{n})\notag\\
&\le\|g\|_{G_d}^2+\frac{4}{n^2}J\le(1+\frac{4m}{n})\|g\|_{G_d}^2=(1+\frac{4m}{n})\|f-A_{m,d}^*(f)\|_{G_d}^2
.\label{3.111}
\end{align}
By Fubini's theorem, \eqref{3.4}, \eqref{3.111}, and \eqref{2.2-0}
we have
\begin{align*}&\quad\ \mathbb{E}\Big(\int_{F_d}\|f- S_{\rm X}^{m}(f)\|_{G_d}^2\mu_d(df)\ \Big|\ \|\widetilde{H}_{m}-I_{m}\|\le
1/2\Big)\\ &=\int_{F_d} \mathbb{E}\Big(\|f- S_{\rm
X}^{m}(f)\|_{G_d}^2\ \Big|\ \|\widetilde{H}_{m}-I_{m}\|\le
1/2\Big)\, \mu_d (df)\\
&=\frac{\int_{F_d}\int_{\|\widetilde{H}_{m}-I_{m}\|\leq\frac{1}{2}}\|f-S_{X}^{m}f\|^{2}_{G_d}
d\mu_{m,d}(x^{1})\ldots d\mu_{m,d}(x^{n})\,\mu_d(df)}{\mathbb{P}(\|\widetilde{H}_{m}-I_{m}\|\leq\frac{1}{2})}\\
&\leq\left(1+\frac{4m}{n}\right)\frac{1}{1-\delta}\,\int_{F_d}\|f-A_{m,d}^*(f)\|_{G_d}^2\mu_d(df)\\
&=\left(1+\frac{4m}{n}\right)\frac{1}{1-\delta}\,(e^{\rm avg}(m,d;
\Lz^{\rm all}))^2.
\end{align*}
By the mean value theorem, we conclude that there are sample
points ${\rm X}^*=\{x^{1*}, \dots, x^{n*}\}$ such that
$\|\widetilde{H}_{m}^*-I_{m}\|\leq\frac{1}{2}$ and
$$\int_{F_d}\|f-S_{\rm X^*}^m(f)\|_{G_d}^2\mu(df)=\mathbb{E}\Big(\int_{F_d}\|f- S_{\rm
X}^{m}(f)\|_{G_d}^2\mu_d(df)\ \Big|\
\|\widetilde{H}_{m}-I_{m}\|\le 1/2\Big).$$ We obtain that
\begin{align*}(e^{\rm avg}(n,d;\Lz^{\rm std}))^2&\le \int_{F_d}\|f-S_{\rm
X^*}^m(f)\|_{G_d}^2\mu(df)\\ &\le
\left(1+\frac{4m}{n}\right)\frac{1}{1-\delta}\,(e^{\rm avg}(m,d;
\Lz^{\rm all}))^2.\end{align*}

 This completes
the proof of Theorem 2.1. $\hfill\Box$

\

We stress that Theorem 2.1 is not constructive since we do not
know how to choose the sample points ${\rm X}^*=\{x^{1*}, \dots,
x^{n*}\}$. We only know that there exist ${\rm X}^*=\{x^{1*},
\dots, x^{n*}\}$ for which the average case error of the weighted
least squares algorithm $S^m_{\rm X^*}$ enjoys the average case
error bound of Theorem 2.1.

\

\noindent{\it Proof of Theorem 2.2.}\

Applying Theorem 2.1 with $\delta=\frac{1}{2^{\sqrt{2}}}$, we
obtain
\begin{equation}
e^{\rm avg}(n,d;\Lz^{\rm
std})\leq\Big(1+\frac{4m}{n}\Big)^{\frac{1}{2}}\Big(\frac{2^{\sqrt{2}}}{2^{\sqrt{2}}-1}\Big)^{\frac12}e^{\rm
avg}(m,d;\Lz^{\rm all}),\label{4.1}
\end{equation}
where $m,n\in\Bbb N$, and $$
m=\Big\lfloor\frac{n}{48\sqrt{2}\ln(4n)}\Big\rfloor.$$ Since
$1+\frac{4m}{n}\leq1+\frac{1}{12\sqrt{2}\ln(4n)}\leq2$, by
\eqref{4.1} we get
\begin{equation}\label{4.01}e^{\rm avg}(n,d;\Lz^{\rm std})\leq 4e^{\rm avg}(m,d;\Lz^{\rm
all}).\end{equation} It follows that
\begin{align}
n^{\star}(\varepsilon,d;\Lambda^{\rm std})&=\min\big\{n\, \big|\,  e^{\rm avg}(n,d;\Lz^{\rm std})\leq\varepsilon{\rm CRI}_d\big\}\nonumber\\
&\leq\min\big\{n\,\big|\, 4e^{\rm avg}(m,d;\Lz^{\rm all})\leq\varepsilon{\rm CRI}_d\big\}\nonumber\\
&=\min\big\{n\mid e^{\rm avg}(m,d;\Lz^{\rm
all})\leq\frac{\varepsilon}{4}{\rm CRI}_d\big\}.\label{4.2}
\end{align}
We note that
$$m=\Big\lfloor\frac{n}{48\sqrt{2}\ln(4n)}\Big\rfloor\geq\frac{n}{48\sqrt{2}\ln(4n)}-1.$$
This inequality is equivalent to
\begin{equation}
4n\leq192\sqrt{2}(m+1)\ln(4n).\label{4.3}
\end{equation}
Taking logarithm on both sides of \eqref{4.3}, and using the
inequality $\ln x\leq\frac{1}{2}x$ for $x\ge 1$, we get
$$\ln(4n)\leq \ln(m+1)+\ln(192\sqrt{2})+\ln\ln(4n),$$and
$$\frac{1}{2}\ln(4n)\leq\ln(4n)-\ln\ln(4n)\leq\ln(m+1)+\ln(192\sqrt{2}).$$
It follows  from \eqref{4.3} that
\begin{equation}
n\leq96\sqrt{2}(m+1)(\ln(m+1)+\ln(192\sqrt{2})).\label{4.4}
\end{equation}
By \eqref{4.2} and \eqref{4.4} we obtain
\begin{align}\label{2.14}n^{\star}(\varepsilon,d;\Lambda^{\rm std})\le  96\sqrt2 \Big(n^{\star}(\frac\varepsilon4,d;\Lambda^{\rm all})+1\Big)
\Big(\ln\big(n^{\star}(\frac\varepsilon4,d;\Lambda^{\rm
all})+1\big)+\ln(192\sqrt{2})\Big). \end{align} Since  for any
$\oz>0$,
\begin{equation*}\label{2.16}\sup_{x\ge1}\frac{96\sqrt2(\ln x +\ln(192\sqrt{2}))}{x^{\oz}}= C_\oz
<+\infty,
\end{equation*}
we obtain \eqref{2.17}.

For sufficiently small $\delta>0$ and
 $m,n\in\mathbb{N}$ satisfying
$$m=\left\lfloor\frac{n}{48(\sqrt{2}\ln(2n)-\ln{\delta})}\right\rfloor,$$
by Theorem 2.1 we have
\begin{align*}
e^{\rm avg}(n,d;\Lz^{\rm std})&\le\Big(1+\frac{4m}{n}\Big)^{\frac12}\frac{1}{\sqrt{1-\delta}}\,e^{\rm avg}(m,d;\Lz^{\rm all})\\
&\le\Big(1+\frac{1}{12\big(\sqrt{2}\ln(2n)+\ln{\frac{1}{\delta}}\big)}\Big)^{\frac12}\frac{1}{\sqrt{1-\delta}}\,e^{\rm avg}(m,d;\Lz^{\rm all})\\
&\le\Big(1+\frac{1}{12\ln{\frac{1}{\delta}}}\Big)^{\frac12}\frac{1}{\sqrt{1-\delta}}\,e^{\rm
avg}(m,d;\Lz^{\rm all})=A_{\delta}\,e^{\rm avg}(m,d;\Lz^{\rm
all}),
\end{align*}
where
$A_{\delta}=\Big(1+\frac{1}{12\ln{\frac{1}{\delta}}}\Big)^{\frac12}\frac{1}{\sqrt{1-\delta}}$.

Using the same method used in the proof of \eqref{4.2}, we have
$$ n^{\star}(\varepsilon,d;\Lambda^{\rm std})\le \min\big\{n\mid e^{\rm avg}(m,d;\Lz^{\rm
all})\leq\frac{\varepsilon}{A_\dz}{\rm CRI}_d\big\}.$$ We note
that
$$n\le48\big(\sqrt{2}\ln(2n)+\ln{\frac{1}{\delta}}\big)(m+1).$$
Taking logarithm on both sides, and using the inequalities $\ln
x\leq\frac{x}{4}$ for $x\ge 9$ and $a+b\le ab$ for $a,b\ge2$, we
get
\begin{align*}
\ln n&\le \ln48+\ln\big(\sqrt{2}\ln(2n)+\ln{\frac{1}{\delta}}\big)+\ln(m+1)\\
&\le \ln48+\ln(\sqrt{2}\ln(2n))+\ln\ln{\frac{1}{\delta}}+\ln(m+1)\\
&\le
\ln48+\frac{\sqrt{2}}{4}\ln(2n)+\ln\ln{\frac{1}{\delta}}+\ln(m+1).
\end{align*}
Since
$$\frac{\sqrt{2}}{4}\ln(2n)\le \ln n-\frac{\sqrt{2}}{4}\ln(2n)\ \ \text{for}\ \ n\ge9,$$
we get
$$\sqrt{2}\ln(2n)\le 4\big(\ln48+\ln\ln{\frac{1}{\delta}}+\ln(m+1)\big).$$
It follows that
$$n\le48\big(4\big(\ln48+\ln\ln{\frac{1}{\delta}}+\ln(m+1)\big)+\ln{\frac{1}{\delta}}\big)(m+1).$$
We conclude that for sufficiently small $\delta>0$,
\begin{align}\label{4.4-1}
n^{\star}(\varepsilon,d;\Lambda^{\rm std}) \le
48\Big(4\big(\ln48&+\ln\ln{\frac{1}{\delta}}+
\ln\big(n^{\star}(\frac{\varepsilon}{A_\delta},d;\Lambda^{\rm
all})+1\big)\big)\\ &+\ln{\frac{1}{\delta}}\Big)
\big(n^{\star}(\frac{\varepsilon}{A_\delta},d;\Lambda^{\rm
all})+1\big).\nonumber
\end{align}

Since for sufficiently small $\oz,\dz>0$, there holds
\begin{equation*}\label{2.16-0}\sup_{x\ge1}\frac{48(4(\ln48+\ln\ln{\frac{1}{\delta}}+ \ln
x)+\ln{\frac{1}{\delta}})}{x^{\oz}}= C_{\oz,\dz} <+\infty,
\end{equation*}we get \eqref{2.18}.

Theorem 2.2 is proved.
   $\hfill\Box$

\section{Equivalence results of algebraic tractability}

First we   consider the equivalences of  ALG-PT and ALG-SPT for
$\Lambda^{\rm std}$ and $\Lambda^{\rm all}$ in the average case
setting. The equivalent results for the normalized error criterion
can be found in \cite{HWW} and \cite[Theorem 24.10]{NW3}.  For the
absolute error criterion, \cite[Theorem 24.11]{NW3} shows the
equivalence of ALG-PT under the condition
\begin{equation}\label{4.11}\Gz_d\le C d^{v} \ \ {\rm for \ all}\ d\in\Bbb N, \ {\rm  some}\ C>0,\ {\rm and  \ some}\ v\ge 0,\end{equation}
and the equivalence of ALG-SPT under the condition \eqref{4.11}
with $v=0$. Xu obtained in \cite[Theorem 3.1]{X5} the equivalence
of ALG-PT under the weaker condition
\begin{equation}\label{4.11-1}\Gz_d\le \exp(C d^{v} )\ \ {\rm for \ all}\ d\in\Bbb N, \ {\rm  some}\ C>0,\ {\rm and  \ some}\ v\ge 0.\end{equation}

 We obtain the following  equivalent
results of ALG-PT and  ALG-SPT  without any condition. Hence, the
condition \eqref{4.11} or \eqref{4.11-1} is unnecessary. This
solves Open Problem 117 as posed by Novak and Wo\'zniakowski in
\cite{NW3}.

\begin{thm}
We consider the problem ${\rm APP}=\{{\rm APP}_d\}_{d\in \Bbb N}$
in the average case   setting for the absolute error criterion.
Then,

$\bullet$   ${\rm ALG}$-${\rm PT}$ for $\Lambda^{\rm all}$  is
equivalent to ${\rm ALG}$-${\rm PT}$
 for $\Lambda^{\rm std}$ .

$\bullet$  ${\rm ALG}$-${\rm SPT}$  for $\Lambda^{\rm all}$ is
equivalent to ${\rm ALG}$-${\rm SPT}$  for $\Lambda^{\rm std}$. In
this case,  the exponents of ${\rm ALG}$-${\rm SPT}$ for
$\Lambda^{\rm all}$ and  $\Lambda^{\rm std}$ are the same.
\end{thm}

\begin{proof} It follows from \eqref{2.4-1} that ALG-PT (ALG-SPT) for $\Lambda^{\rm std}$ means ALG-PT (ALG-SPT) for  $\Lambda^{\rm all}$.
It suffices to show that ALG-PT (ALG-SPT) for $\Lambda^{\rm all}$
 means that ALG-PT (ALG-SPT) for
$\Lambda^{\rm std}$.

Suppose that ALG-PT  holds for $\Lambda^{\rm all}$. Then there
exist $ C\ge 1$ and non-negative $ p,q$ such that
\begin{equation}n^{\rm ABS}(\varepsilon,d;\Lambda^{\rm
all})\leq Cd^{q}\varepsilon^{-p},\ \  \text{for all}\ \
d\in\mathbb{N},\ \varepsilon\in(0,1).\label{4.6-0}\end{equation}
It follows from \eqref{2.17} and \eqref{4.6-0} that
\begin{align*}
n^{\rm  ABS}(\varepsilon,d;\Lambda^{\rm
std})&\leq C_\oz \(Cd^{q}(\frac{\varepsilon}{4})^{-p}+1\)^{1+\oz}\\
&\le C_\oz (2C 4^{p})^{1+\oz}d^{q(1+\oz)}\vz^{-p(1+\oz)},
\end{align*}which means that ALG-PT  holds for  $\Lambda^{\rm std}$.

If ALG-SPT holds  for $\Lambda^{\rm all}$, then \eqref{4.6-0}
holds with $q=0$. We obtain
$$ n^{\rm ABS}(\varepsilon,d;\Lambda^{\rm
std})\le C_\oz(2C 4^{p})^{1+\oz}\vz^{-p(1+\oz)},$$which means that
ALG-SPT  holds for  $\Lambda^{\rm std}$. Furthermore, since $\oz$
can be arbitrary small, we get
\begin{align*} {\rm ALG\!-\!}p^{\rm ABS}(\Lz^{\rm std})\le
{\rm ALG\!-\!}p^{\rm  ABS}(\Lz^{\rm all})\le {\rm ALG\!-\!}p^{\rm
 ABS}(\Lz^{\rm std}),
\end{align*} which means that the exponents of ${\rm ALG}$-${\rm
SPT}$ for $\Lambda^{\rm all}$ and  $\Lambda^{\rm std}$ are the
same. This completes the proof of Theorem 4.1.
\end{proof}

Next we consider the equivalence of  ALG-QPT for $\Lambda^{\rm
std}$ and $\Lambda^{\rm all}$ in the average case setting.
 The result for the
normalized error criterion can be found in \cite[Theorem
24.12]{NW3}. For the absolute error criterion, \cite[Theorem
24.13]{NW3} shows the equivalence of ALG-QPT under the condition
\begin{equation*}\underset{d\to\infty}{\lim \sup}\ \Gz_d<\infty.\end{equation*}
Xu obtained in \cite[Theorem 3.2]{X5} the equivalence of ALG-QPT
under the weaker condition \eqref{4.11-1}.

 We obtain the following  equivalent
results of ALG-QPT  without any condition. Hence, the condition
\eqref{4.11-1}  is unnecessary. This solves Open Problem 118 as
posed by Novak and Wo\'zniakowski in \cite{NW3}.

\begin{thm}
We consider the problem $\rm APP=\{APP_d\}_{d\in\mathbb{N}}$ in
the average case  setting for the absolute error criterion.
 Then, ${\rm ALG}$-${\rm QPT}$ for $\Lambda^{\rm all}$  is
equivalent to ${\rm ALG}$-${\rm QPT}$
 for $\Lambda^{\rm std}$. In this case,  the exponents
 of ${\rm ALG}$-${\rm QPT}$ for
$\Lambda^{\rm all}$ and  $\Lambda^{\rm std}$ are the same.
\end{thm}

\begin{proof}Similar to the proof of Theorem 4.1,
 it is enough to prove that ALG-QPT for
$\Lambda^{\rm all}$ implies ALG-QPT for
 $\Lambda^{\rm std}$.

Suppose that  ALG-QPT holds for $\Lambda^{\rm all}$. Then  there
exist $ C\ge 1$ and non-negative $t$ such that
\begin{equation}\label{4.9}n^{\rm
ABS}(\varepsilon,d;\Lambda^{\rm all})\leq C
\exp(t(1+\ln{d})(1+\ln{\varepsilon^{-1}})),\ \text{for all}\
d\in\mathbb{N},\ \varepsilon\in(0,1).\end{equation}

For sufficiently small $\delta>0$ and $\oz>0$,  it follows from
\eqref{2.18} and \eqref{4.9} that
\begin{align*}
n^{\rm ran,\star}(\varepsilon,d;\Lambda^{\rm std})& \le
C_{\oz,\dz}
 \big(n^{\rm
wor,\star}(\frac{\varepsilon}{A_\delta},d;\Lambda^{\rm
all})+1\big)^{1+\oz}\\
&\leq C_{\oz,\dz}\(C
\exp\big(t(1+\ln{d})\big(1+\ln\big(\frac{\varepsilon}{A_\dz}\big)^{-1})\big)+1\)^{1+\oz}
\\ &\le C_{\oz,\dz}(2C)^{1+\oz}\exp\big(
(1+\oz)t(1+\ln{d})(1+\ln A_\dz +\ln\varepsilon^{-1})\big)\\&\le
C_{\oz,\dz}(2C)^{1+\oz}\exp\big( (1+\oz)t(1+\ln
A_\dz)(1+\ln{d})(1+\ln\varepsilon^{-1})\big),
\end{align*}where $t^*=(1+\oz)(1+\ln A_\dz)t$,
$A_{\delta}=\big(1+\frac{1}{12\ln{\frac{1}{\delta}}}\big)^{\frac12}\frac{1}{\sqrt{1-\delta}}$.
  This implies that ALG-QPT holds for
 $\Lambda^{\rm std}$. Furthermore,
taking the infimum over $t$ for which \eqref{4.9} holds, and
noting that $ \lim\limits_{(\dz,\oz)\to (0,0)}(1+\oz)(1+\ln
A_\dz)=1$,
 we get that
\begin{align*}{\rm ALG\!-\!}t^{\rm  ABS}(\Lz^{\rm std})\le {\rm ALG\!-\!}t^{\rm
 ABS}(\Lz^{\rm all}).
\end{align*}
It follows from \eqref{2.4-1} that
\begin{align*}{\rm ALG\!-\!}t^{\rm  ABS}(\Lz^{\rm std})\le {\rm ALG\!-\!}t^{\rm
 ABS}(\Lz^{\rm all})\le {\rm ALG\!-\!}t^{\rm  ABS}(\Lz^{\rm std}),
\end{align*}
which means that
 the exponents ALG-$t^{\rm  ABS}(\Lz^{\rm all})$
and ALG-$t^{\rm  ABS}(\Lz^{\rm std})$ are equal if ALG-QPT holds
for $\Lambda^{\rm all}$.  This completes the proof of Theorem 4.2.
\end{proof}

Now we   consider the equivalence of  ALG-WT for $\Lambda^{\rm
std}$ and $\Lambda^{\rm all}$  in the average case setting .
 The result for the normalized error criterion can be
found in \cite[Theorem 24.6]{NW3}. For the absolute error
criterion, \cite[Theorem 24.6]{NW3} shows the equivalence of
ALG-WT under the condition
\begin{equation*}\lim_{d\to\infty}\frac{\ln\max(\Gz_d,1)}{d}=0.\end{equation*}

Xu obtained in \cite[Theorem 3.3]{X5} the equivalence of ALG-QPT
under the much weaker condition.
\begin{equation}\label{4.13}\lim_{d\to\infty}\frac{\ln\big(1+\ln\max(\Gz_d,1)\big)}{d}=0.\end{equation}

 We obtain the following  equivalent
results of ALG-WT  without any condition. Hence, the condition
\eqref{4.13}  is unnecessary. This solves Open Problem 116 as
posed by Novak and Wo\'zniakowski in \cite{NW3}.

\begin{thm}
We consider the problem $\rm APP=\{APP_d\}_{d\in\mathbb{N}}$  in
the average case setting  for the absolute error criterion. Then,
${\rm ALG}$-${\rm WT}$ for $\Lambda^{\rm all}$  is equivalent to
${\rm ALG}$-${\rm WT}$
 for $\Lambda^{\rm std}$.
\end{thm}
\begin{proof} The proof is identical to the proof of Theorem 4.4 with $s = t = 1$ for the absolute error criterion. We omit the details. \end{proof}

Finally, we   consider the equivalences of ALG-$(s, t)$-WT and
ALG-UWT for $\Lambda^{\rm std}$ and $\Lambda^{\rm all}$  in the
average case setting.
 The results for the normalized error criterion can be
found in \cite[Theorems 3.4 and 3.5]{X5}. For the absolute error
criterion, \cite[Theorem 3.4]{X5} shows the equivalence  of
ALG-$(s, t)$-WT  under the condition
\begin{equation}\label{4.14}\lim_{d\to\infty}\frac{\ln\big(1+\ln\max(\Gz_d,1)\big)}{d^t}=0.\end{equation}
\cite[Theorem 3.5]{X5} shows the equivalence  of ALG-UWT under the
condition
\begin{equation}\label{4.15}\lim_{d\to\infty}\frac{\ln\big(1+\ln\max(\Gz_d,1)\big)}{d^t}=0\  \ {\rm for \ all\ } t>0.\end{equation}

 We obtain the following  equivalent
results of  ALG-$(s, t)$-WT for
 fixed $s,t>0$ and ALG-UWT  for the absolute error
criterion without any condition. Hence, the condition \eqref{4.14}
or \eqref{4.15} is unnecessary.

\begin{thm}
We consider the problem $\rm APP=\{APP_d\}_{d\in\mathbb{N}}$  in
the average case setting  for the absolute  error criterion. Then
for fixed $s,t>0$,  ${\rm ALG}$-$(s,t)$-${\rm WT}$ for
$\Lambda^{\rm all}$ is equivalent to ${\rm ALG}$-$(s,t)$-${\rm
WT}$
 for $\Lambda^{\rm std}$.
\end{thm}

\begin{proof}  Again
 it is enough to prove that ${\rm ALG}$-$(s,t)$-${\rm WT}$ for
$\Lambda^{\rm all}$  implies ${\rm ALG}$-$(s,t)$-${\rm WT}$ for
$\Lambda^{\rm std}$.

Suppose that ${\rm ALG}$-$(s,t)$-${\rm WT}$ holds for
$\Lambda^{\rm all}$. Then  we have
\begin{equation}
\lim_{\varepsilon^{-1}+d\rightarrow\infty}\frac{\ln n^{\rm
ABS}(\varepsilon,d;\Lambda^{\rm
all})}{\varepsilon^{-s}+d^{t}}=0.\label{4.10}
\end{equation}
It follows from \eqref{2.17}  that for $\oz>0$,
\begin{align*}
\frac{\ln n^{\rm ABS}(\varepsilon,d;\Lambda^{\rm
std})}{\varepsilon^{-s}+d^{t}} &\leq\frac{\ln
\Big(C_\oz\big(n^{\rm ABS}(\varepsilon/4,d;\Lambda^{\rm
all})+1\big)^{1+\oz}\Big)}{\varepsilon^{-s}+d^{t}}\\ &\le
\frac{\ln
(C_\oz2^{1+\oz})}{\varepsilon^{-s}+d^{t}}+\frac{4^s(1+\oz)\,\ln
n^{\rm ABS}(\varepsilon/4,d;\Lambda^{\rm
all})}{(\varepsilon/4)^{-s}+d^{t}}.
\end{align*}
Since $\varepsilon^{-1}+d\rightarrow\infty$ is equivalent to
$\vz^{-s}+d^t\to \infty$,  by \eqref{4.10} we get that
$$\lim\limits_{\varepsilon^{-1}+d\rightarrow\infty}\frac{\ln (C_\oz2^{1+\oz})}{\varepsilon^{-s}+d^{t}}=0\ \ \ {\rm and} \ \
\lim_{\varepsilon^{-1}+d\rightarrow\infty}\frac{\ln n^{\rm
ABS}(\varepsilon/4,d;\Lambda^{\rm
all})}{(\varepsilon/4)^{-s}+d^{t}}=0.$$We  obtain
$$\lim\limits_{\varepsilon^{-1}+d\rightarrow\infty} \frac{\ln n^{\rm ABS}(\varepsilon,d;\Lambda^{\rm
std})}{\varepsilon^{-s}+d^{t}}=0,$$ which implies
 ${\rm ALG}$-$(s,t)$-${\rm WT}$ for
$\Lambda^{\rm std}$.
 The proof of Theorem 4.4 is finished.
\end{proof}

\begin{thm}
We consider the problem $\rm APP=\{APP_d\}_{d\in\mathbb{N}}$  in
the average case setting  for the absolute  error criterion. Then
 ${\rm ALG}$-${\rm UWT}$  for
$\Lambda^{\rm all}$  is equivalent to  ${\rm ALG}$-${\rm UWT}$
 for $\Lambda^{\rm std}$.
\end{thm}
\begin{proof}By definition we know  that ${\rm APP}$ is ALG-UWT if and only if
 ${\rm APP}$ is ALG-$(s,t)$-WT for all $s,t>0$. Since by Theorem 4.4 ${\rm
ALG}$-$(s,t)$-${\rm WT}$ for $\Lambda^{\rm std}$ is equivalent to
${\rm ALG}$-$(s,t)$-${\rm WT}$ for $\Lambda^{\rm all}$ for all
$s,t>0$, we get the equivalence of ${\rm ALG}$-${\rm UWT}$ for
$\Lambda^{\rm std}$ and $\Lambda^{\rm all}$. Theorem 4.5 is
proved.
\end{proof}

\noindent{\it Proof of Theorem 2.3.} \

Theorem 2.3 follows from Theorems 4.1-4.5 immediately.
$\hfill\Box$

\section{Equivalence results of exponential  tractability}

First we consider exponential convergence. Assume that there exist
two constants $A\ge1$ and $q\in(0,1)$ such that
\begin{equation}\label{5.01}\sqrt{\lz_{n,d}} \le
Aq^{n}\,e^{\rm avg} (0,d;\Lambda^{\rm
all})=Aq^{n}\sqrt{\Gz_d}.\end{equation}It follows that
$$e^{\rm avg} (n,d;\Lambda^{\rm
all})\le \frac{A}{1-q}q^{n+1}\sqrt{\Gz_d}. $$
 Novak
and Wo\'zniakowski proved in \cite[Corollary 24.5]{NW3} that there
exist two constants $C_1\ge1$ and $q_1\in (q,1)$ independent of
$d$ and $n$ such that
\begin{equation}\label{5.02} e^{\rm
avg} (n,d;\Lambda^{\rm std})\le \frac{ C_1 A}{1-q}\, q_1^{\sqrt
{n}}\,\sqrt{\Gz_d}\,.\end{equation} If  $A, q$ in \eqref{5.01} are
independent of $d$, then $$n^{\rm NOR}(\varepsilon,d;\Lambda^{\rm
all})\leq C_2(\ln\varepsilon^{-1}+1),$$ and $$n^{\rm
NOR}(\varepsilon,d;\Lambda^{\rm std})\leq
C_3(\ln\varepsilon^{-1}+1)^2.$$  Novak and Wo\'zniakowski posed
the  following Open Problem 115:

(1) Verify if the upper bound in \eqref{5.02} can be improved.

 (2) Find the smallest $p$ for which there holds
$$n^{\rm
NOR}(\varepsilon,d;\Lambda^{\rm std})\leq
C_4(\ln\varepsilon^{-1}+1)^p.$$ We know that $p\le 2$, and if
\eqref{5.01} is sharp then $p\ge1$.

The following theorem  gives a confirmative solution to  Open
Problem 115 (1). We improve enormously the upper bound $q_1^{\sqrt
n}$ in \eqref{5.02} to $q_2^{\frac{n}{\ln(4n)}}$ in \eqref{5.05},
where $q_1,q_2\in(q,1)$.

\begin{thm} Let $m,n\in\Bbb N$ and \begin{equation}\label{5.03}
m=\Big\lfloor\frac{n}{48\sqrt{2}\ln(4n)}\Big\rfloor.\end{equation}
Then we have \begin{equation}\label{5.04}e^{\rm avg}(n,d;\Lz^{\rm
std})\leq 4e^{\rm avg}(m,d;\Lz^{\rm all}).\end{equation}
 Specifically, if \eqref{5.01} holds, then we have
  \begin{equation}\label{5.05}e^{\rm avg}(n,d;\Lz^{\rm
std})\leq \frac{4 A}{1-q} q_2^{\frac{n}{\ln(4n)}} e^{\rm
avg}(0,d;\Lz^{\rm all})\, ,\end{equation}where
$q_2=q^{\frac1{48\sqrt2}}\in(q,1)$.
\end{thm}
\begin{proof} Inequality \eqref{5.04} is just \eqref{4.01}, which
has been  proved. If \eqref{5.01} holds, then by \eqref{5.03} and
\eqref{5.04} we get \begin{align*} e^{\rm avg}(n,d;\Lz^{\rm
std})&\leq \frac{4 A}{1-q} \,
q^{\big\lfloor\frac{n}{48\sqrt{2}\ln(4n)}\big\rfloor+1} \sqrt{\Gz_d}\\
&\le \frac{4 A}{1-q} q^{\frac{n}{48\sqrt{2}\ln(4n)}}
\sqrt{\Gz_d}\\ &=\frac{4 A}{1-q} q_2^{\frac{n}{\ln(4n)}} e^{\rm
avg}(0,d;\Lz^{\rm all}).\end{align*} This completes the proof of
Theorem 5.1.
\end{proof}

Now  we   consider the equivalences of various notions of
exponential tractability  for $\Lambda^{\rm std}$ and
$\Lambda^{\rm all}$ for the absolute error criterion in the
average case setting.

First we   consider the equivalences of  EXP-PT and EXP-SPT for
$\Lambda^{\rm std}$ and $\Lambda^{\rm all}$. The results for the
normalized error criterion can be found in \cite[Theorem 4.1]{X5}.
For the absolute error criterion, \cite[Theorem 4.1]{X5} shows the
equivalences  of EXP-PT and EXP-SPT under the condition
\eqref{4.11-1}.

 We obtain the following  equivalent
results of EXP-PT and  EXP-SPT  without any condition.

\begin{thm}
We consider the problem $\rm APP=\{APP_d\}_{d\in\mathbb{N}}$  in
the average case setting. Then

$\bullet$ for the absolute error criterion,  ${\rm EXP}$-${\rm
PT}$ (${\rm EXP}$-${\rm SPT}$) for $\Lambda^{\rm all}$  is
equivalent to ${\rm EXP}$-${\rm PT}$ (${\rm EXP}$-${\rm SPT}$)
 for $\Lambda^{\rm std}$;

$\bullet$  if ${\rm EXP}$-${\rm SPT}$ holds for $\Lambda^{\rm
all}$
 for the absolute or normalized  error criterion,   then the
exponents of ${\rm EXP}$-${\rm SPT}$ for $\Lambda^{\rm all}$ and
$\Lambda^{\rm std}$ are the same.
\end{thm}
\begin{proof}
Again, it is enough to prove that EXP-PT for $\Lambda^{\rm all}$
implies EXP-PT for $\Lambda^{\rm std}$ for the absolute error
criterion.

Suppose that EXP-PT  holds for $\Lambda^{\rm all}$. Then there
exist $C\ge1 $ and non-negative $ p,q$ such that
\begin{equation}\label{5.1}n^{\rm
ABS}(\varepsilon,d;\Lambda^{\rm all})\leq
Cd^{q}(\ln\varepsilon^{-1}+1)^{p},\ \text{for all}\
d\in\mathbb{N},\ \varepsilon\in(0,1).\end{equation}
 It follows from
\eqref{2.17} and \eqref{5.1} that
\begin{align*}
n^{\rm ABS}(\varepsilon,d;\Lambda^{\rm
std})&\leq C_\oz \(Cd^{q}(\ln(\frac{\varepsilon}{4})^{-1}+1)^p+1\)^{1+\oz}\\
&\le C_\oz (2C)^{1+\oz} (1+\ln
4)^{p(1+\oz)}d^{q(1+\oz)}(\ln\varepsilon^{-1}+1)^{p(1+\oz)},
\end{align*}which means that EXP-PT  holds for  $\Lambda^{\rm std}$.

If EXP-SPT holds  for $\Lambda^{\rm all}$, then \eqref{5.1} holds
with $q=0$. We obtain
$$ n^{\rm ABS}(\varepsilon,d;\Lambda^{\rm
std})\le C_\oz (2C)^{1+\oz} (1+\ln
4)^{p(1+\oz)}(\ln\varepsilon^{-1}+1)^{p(1+\oz)},$$which means that
EXP-SPT  holds for  $\Lambda^{\rm std}$. Furthermore, if ${\rm
EXP}$-${\rm SPT}$ holds for $\Lambda^{\rm all}$
 for the absolute or normalized  error criterion and $p^*={\rm
EXP\!-\!}p^{\star}(\Lz^{\rm all})$ for $\star\in\{{\rm
ABS,\,NOR}\}$, then for any $\vz>0$, there is a constant $C_\vz\ge
1$ for which
$$ n^{\star}(\varepsilon,d;\Lambda^{\rm
all})\le C_\vz  (\ln\varepsilon^{-1}+1)^{p^*+\vz}$$ holds. Using
the same method, we get
$$ n^{\star}(\varepsilon,d;\Lambda^{\rm
std})\le C_\oz (2C_\vz)^{1+\oz} (1+\ln
4)^{(p^*+\vz)(1+\oz)}(\ln\varepsilon^{-1}+1)^{(p^*+\vz)(1+\oz)},$$
Noting that $\vz,\oz$ can be arbitrary small, we have for
$\star\in\{{\rm ABS,\,NOR}\}$,
\begin{align*} {\rm
EXP\!-\!}p^{\star}(\Lz^{\rm std})\le  {\rm
EXP\!-\!}p^{\star}(\Lz^{\rm all})\le {\rm
EXP\!-\!}p^{\star}(\Lz^{\rm std}),
\end{align*} which means that the exponents of ${\rm EXP}$-${\rm
SPT}$ for $\Lambda^{\rm all}$ and  $\Lambda^{\rm std}$ are the
same. This completes the proof of Theorem 5.2.
\end{proof}

\begin{rem} We remark that if \eqref{5.01} holds with $A, q$ independent of $d$, then the problem {\rm APP} is {\rm
EXP-SPT} for $\Lambda^{\rm all}$  in the average case setting  for
the normalized error criterion, and the exponent ${\rm
EXP\!-\!}p^{\rm NOR}(\Lz^{\rm all})\le 1$. If \eqref{5.01} is
sharp, then ${\rm EXP\!-\!}p^{\rm  NOR}(\Lz^{\rm all})= 1$.

  Open Problem 115 (2) is equivalent to finding the exponent $ {\rm
EXP\!-\!}p^{\rm  NOR}(\Lz^{\rm std})$ of ${\rm EXP}$-${\rm SPT}$.
By Theorem 5.2 we obtain that if \eqref{5.01} holds, then $ {\rm
EXP\!-\!}p^{\rm NOR}(\Lz^{\rm std})\le 1,$ and if \eqref{5.01} is
sharp, then $ {\rm EXP\!-\!}p^{\rm  NOR}(\Lz^{\rm std})=1$.

This solves Open Problem 115 (2) as posed by Novak and
Wo\'zniakowski in  \cite{NW3}.
\end{rem}

Next we consider the equivalence of  EXP-QPT for $\Lambda^{\rm
std}$ and $\Lambda^{\rm all}$ in the average case setting.
 The result for the normalized error criterion can be found in \cite[Theorem
4.2]{X5}. For the absolute error criterion, \cite[Theorem 4.2]{X5}
shows the equivalence of EXP-QPT under the condition
\eqref{4.11-1}.

 We obtain the following  equivalent
results of EXP-QPT  without any condition.

\begin{thm}
We consider the problem $\rm APP=\{APP_d\}_{d\in\mathbb{N}}$  in
the average case setting.
 Then, for the absolute error criterion ${\rm EXP}$-${\rm QPT}$ for $\Lambda^{\rm all}$  is equivalent to ${\rm EXP}$-${\rm QPT}$ for $\Lambda^{\rm std}$.
 If ${\rm EXP}$-${\rm QPT}$ holds for $\Lambda^{\rm
all}$ for the absolute or normalized  error criterion, then the
exponents
 of ${\rm EXP}$-${\rm QPT}$ for
$\Lambda^{\rm all}$ and  $\Lambda^{\rm std}$ are the same.
\end{thm}

\begin{proof}Again,
 it is enough to prove that EXP-QPT for
$\Lambda^{\rm all}$ implies EXP-QPT for
 $\Lambda^{\rm std}$ for the absolute error criterion.

Suppose that  EXP-QPT holds for $\Lambda^{\rm all}$ for the
absolute or normalized  error criterion. Then  there exist $ C\ge
1$ and non-negative $t$ such that for $\star\in\{{\rm
ABS,\,NOR}\}$,
\begin{equation}\label{5.2}n^{\star}(\va ,d;\Lz^{\rm all})\leq C
\exp(t(1+\ln{d})(1+\ln(\ln\varepsilon^{-1}+1))),\ \text{for all}\
d\in\mathbb{N},\ \varepsilon\in(0,1).\end{equation} For
sufficiently small  $\oz>0$ and $\delta>0$, it follows from
\eqref{2.18} and \eqref{5.2} that
\begin{align}
n^{\star}(\varepsilon,d;\Lambda^{\rm std}) & \le C_{\oz,\dz}
 \big(n^{\star}(\frac{\varepsilon}{A_\delta},d;\Lambda^{\rm
all})+1\big)^{1+\oz}\notag \\&\leq C_{\oz,\dz}\(C
\exp\big(t(1+\ln{d})\big(1+\ln(\ln\varepsilon^{-1}+\ln
A_\dz+1))\big)+1\)^{1+\oz} \notag \\&\le
C_{\oz,\dz}(2C)^{1+\oz}\exp\big( (1+\oz)t(1+\ln{d})(1+\ln (\ln
A_\dz+1) +\ln(\ln\varepsilon^{-1}+1))\big)\notag \\&\le
C_{\oz,\dz}(2C)^{1+\oz}\exp\big(
t^*(1+\ln{d})(1+\ln(\ln\varepsilon^{-1}+1))\big),\label{5.2-10}
\end{align}where $t^*=(1+\oz)(1+\ln(\ln A_\dz+1))t$ and
$A_{\delta}=\big(1+\frac{1}{12\ln{\frac{1}{\delta}}}\big)^{\frac12}\frac{1}{\sqrt{1-\delta}}$,
in the third inequality we used the fact $$\ln (1+a+b)\le
\ln(1+a)+\ln (1+b),\ \ \ a,b\ge 0.$$ The  inequality
\eqref{5.2-10} with $\star={\rm ABS}$ implies that EXP-QPT holds
for
 $\Lambda^{\rm std}$ for the absolute  error criterion.

Next, we suppose that ${\rm EXP}$-${\rm QPT}$ holds for
$\Lambda^{\rm all}$ for the absolute or normalized  error
criterion.   Taking the infimum over $t$ for which \eqref{5.2}
holds, and noting that $$ \lim\limits_{(\dz,\oz)\to
(0,0)}(1+\oz)(1+\ln(\ln A_\dz+1))=1, $$ by \eqref{2.4-1} we obtain
that
\begin{align*}  {\rm EXP\!-\!}t^{\star}(\Lz^{\rm
std})\le{\rm EXP\!-\!}t^{\star}(\Lz^{\rm all})\le {\rm
EXP\!-\!}t^{\star}(\Lz^{\rm std}).\end{align*}
which means that
 the exponents EXP-$t^{\star}(\Lz^{\rm all})$
and EXP-$t^{\star}(\Lz^{\rm std})$ are equal.  This completes the
proof of Theorem 5.4.
\end{proof}

Next, we   consider the equivalences of EXP-$(s, t)$-WT and EXP-WT
for $\Lambda^{\rm std}$ and $\Lambda^{\rm all}$  in the average
case setting. The results for the normalized error criterion can
be found in \cite[Theorems 4.3 and 4.4]{X5}. For the absolute
error criterion, \cite[Theorem 4.3]{X5} shows the equivalence  of
EXP-WT  under the condition\eqref{4.13}. Meanwhiles, \cite[Theorem
4.4]{X5} shows the equivalence  of EXP-$(s, t)$-WT under the
condition \eqref{4.14}.

 We obtain the
following equivalent results of EXP-$(s, t)$-WT  and
EXP-WT for the absolute  error criterion without any condition.

\begin{thm}
We consider the problem $\rm APP=\{APP_d\}_{d\in\mathbb{N}}$  in
the average case setting  for the absolute error
criterion. Then for
 fixed $s,t>0$, ${\rm EXP}$-$(s,t)$-${\rm WT}$ for $\Lambda^{\rm all}$  is
equivalent to ${\rm EXP}$-$(s,t)$-${\rm WT}$
 for $\Lambda^{\rm std}$. Specifically, ${\rm EXP}$-${\rm WT}$ for $\Lambda^{\rm all}$  is
equivalent to ${\rm EXP}$-${\rm WT}$
 for $\Lambda^{\rm std}$.
\end{thm}

\begin{proof}
Again,
 it is enough to prove that ${\rm EXP}$-$(s,t)$-${\rm WT}$ for
$\Lambda^{\rm all}$  implies ${\rm EXP}$-$(s,t)$-${\rm WT}$ for
$\Lambda^{\rm std}$.

Suppose that ${\rm EXP}$-$(s,t)$-${\rm WT}$ holds for
$\Lambda^{\rm all}$. Then  we have
\begin{equation}
\label{5.3}\lim_{\varepsilon^{-1}+d\rightarrow\infty}\frac{\ln
n^{\rm ABS}(\varepsilon,d;\Lambda^{\rm
all})}{(1+\ln\varepsilon^{-1})^{s}+d^{t}}=0.
\end{equation}
It follows from \eqref{2.17}  that for $\oz>0$,
\begin{align*}
&\quad\ \frac{\ln n^{\rm ABS}(\varepsilon,d;\Lambda^{\rm
std})}{(1+\ln\varepsilon^{-1})^{s}+d^{t}} \leq\frac{\ln
\Big(C_\oz\big(n^{\rm ABS}(\varepsilon/4,d;\Lambda^{\rm
all})+1\big)^{1+\oz}\Big)}{(1+\ln\varepsilon^{-1})^{s}+d^{t}}\\
&\le \frac{\ln
(C_\oz2^{1+\oz})}{(1+\ln\varepsilon^{-1})^{s}+d^{t}}+\frac{(1+\ln
4)^s(1+\oz)\,\ln n^{\rm ABS}(\varepsilon/4,d;\Lambda^{\rm
all})}{(1+\ln(\varepsilon/4)^{-1})^{s}+d^{t}}.
\end{align*}
Since $\varepsilon^{-1}+d\rightarrow\infty$ is equivalent to
$(1+\ln\varepsilon^{-1})^{s}+d^{t}\to \infty$,  by \eqref{5.3} we
get that
$$\lim\limits_{\varepsilon^{-1}+d\rightarrow\infty}\frac{\ln (C_\oz2^{1+\oz})}{(1+\ln\varepsilon^{-1})^{s}+d^{t}}=0\ \ \ {\rm and} \ \
\lim_{\varepsilon^{-1}+d\rightarrow\infty}\frac{\ln n^{\rm
ABS}(\varepsilon/4,d;\Lambda^{\rm
all})}{(1+\ln(\varepsilon/4)^{-1})^{s}+d^{t}}=0.$$We  obtain
$$\lim\limits_{\varepsilon^{-1}+d\rightarrow\infty} \frac{\ln n^{\rm ABS}(\varepsilon,d;\Lambda^{\rm
std})}{(\ln\varepsilon^{-1})^{s}+d^{t}}=0,$$ which implies that
 ${\rm EXP}$-$(s,t)$-${\rm WT}$ holds for
$\Lambda^{\rm std}$.

Specifically, EXP-WT is just  ${\rm EXP}$-$(s,t)$-${\rm WT}$ with
$s=t=1$.

This completes the proof of
 Theorem 5.5.
\end{proof}

Finally, we   consider the equivalences of EXP-UWT  for
$\Lambda^{\rm std}$ and $\Lambda^{\rm all}$  in the average case
setting. The results for the normalized error criterion can be
found in \cite[Theorems 4.5]{X5}. For the absolute error
criterion, \cite[Theorem 4.5]{X5} shows the equivalence  of
EXP-UWT under the condition \eqref{4.15}.

 We obtain the
following equivalent result of EXP-UWT for the absolute  error
criterion without any condition.

\begin{thm}
We consider the problem $\rm APP=\{APP_d\}_{d\in\mathbb{N}}$  in
the average case setting  for the absolute  error criterion. Then,
${\rm EXP}$-${\rm UWT}$ for $\Lambda^{\rm all}$ is equivalent to
${\rm EXP}$-${\rm UWT}$
 for $\Lambda^{\rm std}$.
\end{thm}

\begin{proof}
By definition we know  that ${\rm APP}$ is EXP-UWT if and only if
 ${\rm APP}$ is EXP-$(s,t)$-WT for all $s,t>0$. Then Theorem 5.6 follows from Theorem 5.5 immediately.
\end{proof}

\noindent{\it Proof of Theorem 2.4.}  \

Theorem 2.4 follows from Theorems 5.2 and 5.4-5.6 immediately.
$\hfill\Box$

\

 \noindent{\bf Acknowledgment}  This work was supported by the
National Natural Science Foundation of China (Project no.
11671271).

\end{document}